\documentclass[11pt]{amsart}
\makeatletter

\@addtoreset{equation}{section}
\makeatother
\usepackage{comment}
\usepackage{amsmath, amssymb}
\usepackage[T1]{fontenc}

\newcommand{\ord}{\mathrm{ord}}

\numberwithin{equation}{section}

\renewcommand{\epsilon}{\varepsilon}

\newtheorem{thm}[equation]{Theorem}

\newtheorem{lma}[equation]{Lemma}

\newtheorem{prop}[equation]{Proposition}

\newtheorem{ques}[equation]{Question}

\theoremstyle{definition}

\newtheorem*{rem*}{Remark}
\newtheorem*{note*}{Notation}

\begin{document}

\title[Trascendence of values of iterated exponentials]{Transcendence of values of the iterated\\ exponential function at algebraic points}

\author[H. Kobayashi]{ Hirotaka Kobayashi }
\address{Hirotaka Kobayashi\\
Graduate School of Mathematics\\ Nagoya University\\ Furocho\\ Chikusa-ku\\ Nagoya\\ 464-8602\\ Japan }
\curraddr{}
\email{m17011z@math.nagoya-u.ac.jp}

\author[K. Saito]{ Kota Saito }
\address{Kota Saito\\
Graduate School of Mathematics\\ Nagoya University\\ Furocho\\ Chikusa-ku\\ Nagoya\\ 464-8602\\ Japan }
\curraddr{}
\email{m17013b@math.nagoya-u.ac.jp}

\author[W. Takeda]{ Wataru Takeda }
\address{Wataru Takeda\\
Graduate School of Mathematics\\ Nagoya University\\ Furocho\\ Chikusa-ku\\ Nagoya\\ 464-8602\\ Japan }
\curraddr{}
\email{d18002r@math.nagoya-u.ac.jp}



\begin{abstract}
We say that the order of an algebraic number $A$ is the minimum of positive integers $k$ such that $A^k$ is rational. In this paper, we show that the number of algebraic numbers $A$ with order $k$ such that
\[
A,\ A^A,\ A^{A^A},\ \ldots
\]
converges to an algebraic number is approximated by $(e-1/e)\varphi (k)$. Here $\varphi(k)$ denotes Euler's totient function.
\end{abstract}

\maketitle

\section{Introduction}\label{section1}

We say that a complex number $\alpha$ is algebraic if there exists a polynomial $f(X)$ with rational coefficients such that $f(\alpha)=0$, and $\alpha$ is transcendental if $\alpha$ is not algebraic. It is a great topic to clarify the transcendence (or not transcendence) of a given number in transcendental number theory. The first example of a transcendental number was constructed by Liouville in 1844 \cite{Liouville}. He proved that $\sum_{n\geq 0}10^{-n!}$ is transcendental. At the 1900 International Congress of Mathematicians in Paris, Hilbert posed unsolved 23 problems at the time, which are called Hilbert's 23 problems. One of the problems is on transcendental number theory. This is called Hilbert's 7th problem as follows:

``{\it Prove that the expression $\alpha^\beta$, for an algebraic base $\alpha\neq 0,1$ and an irrational algebraic exponent $\beta$, e.g., the number $2^{\sqrt{2}}$ or $e^\pi=i^{-2i}$, always represents a transcendental or at least an irrational number.}''

In 1934, Gelfond and Schneider solved this problem, independently.
  
\begin{thm}[Gelfond-Schneider\cite{Gelfond1,Gelfond2, Schneider1, Schneider2}]\label{lma2}
If $\alpha \in\mathbb{A}$ and $\beta\in \mathbb{A}\setminus \mathbb{Q}$, then $\alpha^\beta$ is transcendental. 
\end{thm}

We use this result without proof.

In this paper, we study the transcendence of the limit of a sequence 
\begin{equation}
\label{IE}
	 x,\quad x^x, \quad x^{x^x},\ldots.
\end{equation}
This limit is denoted by $h(x)$, and called the iterated exponential function. Formally, $h(x)$ can be written as
 \[
 h(x)=x^{x^{x^{\cdot^{\cdot^{\cdot}}} }}.
\]
The limit of a sequence (\ref{IE}) is convergent for every  $e^{-e}\leq x\leq e^{1/e}$ from the result of Barrow \cite[Theorem~5]{Barrow}, and he also proved that for every $e^{-e}\leq x\leq e^{1/e}$,   
\begin{equation}
\label{h1}
 h(x)=x^{h(x)}, \quad 1/e\leq h(x)\leq e. 
\end{equation}
Then we propose the following question:
\begin{ques} 
Is $h(A)$ transcendental or not for a given algebraic number $A$ if h(A) is convergent?  
\end{ques}
For some algebraic number $A$, the transcendence of $h(A)$ has already been known by the result of Sondow and Marques.  
\begin{prop}[{\cite[Corollary~4.2]{SondowMarques}}]\label{Prop1}
 Let $A\in [e^{-e},e^{1/e}]$. If either 
\begin{itemize}
\item[(i)] $A^n\in\mathbb{A}\setminus \mathbb{Q} $ for all $n\in \mathbb{N}$, or 
\item[(ii)] $A\in\mathbb{Q}\setminus \{1/4,1\}$,
\end{itemize}
 then $h(A)$ is transcendental, where $\mathbb{A}$ denotes the set of all algebraic numbers. 
\end{prop} 
 However, they did not mention the case when there exists $n\in \mathbb{N}$ such that $A^n\in \mathbb{Q}$. This paper gives a new result in this unknown case.
 
A complex valued function $f(x)$ is called transcendental, if there exists no polynomial $P(y)$ with $\mathbb{C}(x)$ coefficients such that $P(f(x))\equiv0$. It is known that there are entire transcendental functions $f$ such that $f(\alpha)$ is an algebraic number for any algebraic number $\alpha$ \cite{va}.
For transcendental functions $f$, the exceptional set is defined as 
\[\{\alpha\in\mathbb{A}~|~f(\alpha)\in\mathbb{A}\}.\]
There are not many studies on this set. In this paper, we also consider the exceptional set for the iterated exponential function.
 
\begin{note*} In this paper, $\mathbb{N}$ denotes the set of all positive integers, $\mathbb{Z}$ denotes the set of all integers, $\mathbb{Q}$ denotes the set of all rational numbers, $\mathbb{R}$ denotes the set of all real numbers, $\mathbb{A}$ denotes the set of all algebraic numbers, $\mathbb{T}$ denotes the set of all transcendental numbers, $\gcd(a,b)$ denotes the greatest common divisor of integers $a$ and $b$, $a\mid b$ denotes that $b$ can be divided by $a$, and $p^k\parallel a$ denotes that $p^k\mid a$ and $p^{k+1}\nmid a$.
\end{note*}

\section{Results}\label{section2} 

We now define the function $\ord : \mathbb{A}\rightarrow \mathbb{N}\cup\{\infty\}$ as 
\[
\ord (A)=\min\{n\in\mathbb{N} \: :\:  A^n\in\mathbb{\mathbb{Q}}\} 
\]
if there exists $n\in \mathbb{N}$ such that $A^n\in \mathbb{Q}$, and define $\ord(A)=\infty$ otherwise.  We say that $\ord (A)$ is the {\it order} of an algebraic number $A$. The first goal of this paper is to prove the following theorem.

\begin{thm}\label{main1}
 Fix an integer $k\geq 2$. For every $A\in \mathbb{A}\cap [e^{-e},e^{1/e}]$ with $\ord(A)=k$, $h(A)$ is transcendental except for $A\in \mathcal{E}(k)$, where let
 \[
 \theta:=(\log2-1/e)^{-1}=3.074390\cdots,
 \]
 and
 \[
 \mathcal{E}(k):=\left\{\left(\frac{kt}{s}\right)^{\frac{s}{kt} }\ :\begin{array}{l}
 \ 1\leq t\leq \theta\log k,\ kt/e\leq s\leq kte,\ s,t \in\mathbb{N},\\ (kt)^{1/t}, s^{1/t}\in\mathbb{N},\ \gcd(kt,s)=1 
 \end{array}
 \right\}.
 \]
 \end{thm}

We show Theorem~\ref{main1} in Section~\ref{section4}. Moreover, for the case that $k$ is a square-free integer, we can characterize the set of all algebraic numbers $A$ of order $k$ satisfying that $h(A)$ is algebraic.
\begin{thm}\label{th:main2nd}
If $k\geq 3$ is square free, then
\[\left\{A\in \mathbb{A}\cap [e^{-e},e^{1/e}] :\begin{array}{l} h(A) \text{ is algebraic,}\\
\ord (A)=k
\end{array}\right\}=\left\{\left(\frac{k}{s}\right)^{\frac{s}{k}}: \begin{array}{l} k/e\leq s\leq ke,\\ 
\gcd(s,k)=1  
\end{array}\right\}.\]
\end{thm}

\begin{rem*}
We also get the result for $k=2$. The explicit form is stated after the proof of Theorem \ref{th:main2nd}. We do not know whether the set $\mathcal{E}(k)$ is equal to the set of all algebraic numbers $A$ with order $k$ such that $h(A)$ is algebraic. As we discussed in Section~\ref{section1}, the case $k=1$ or $\infty$ has been already proved by Sondow and Marques (Proposition~\ref{Prop1}). We define $\mathcal{E}(1)=\{1/4,1\}$ and $\mathcal{E}(\infty)=\emptyset$. From Theorem~\ref{main1} and Proposition~\ref{Prop1}, $h(A)$ is transcendental except for $A\in \mathcal{E}(k)$ for every $k\in \mathbb{N}\cup \{\infty\}$. 
\end{rem*}

It is clear that $\mathcal{E}(k)$ is a finite set for every $k\geq 1$. Thus we can define an arithmetic function $Q(k)$ as 
\[
Q(k)=\#\{A\in \mathbb{A}\cap [e^{-e}, e^{1/e}]\: :\: h(A) \text{ is algebraic, and } \ord(A)=k \},
\]
where $\# X$ denotes the cardinality of $X$ for every finite set $X$. 

For every functions $f(k),g(k)$ and for every non-negative function $h(k)$, we define $f(k)=g(k)+O(h(k))$ as there exists some constant $C>0$ such that $|f(k)-g(k)|\leq C h(k)$. Let $\varphi(k)$ be the number of positive integers up to a given integer $k$ that are relatively prime to $k$, which is called Euler's totient function. We find that $Q(k)$ can be approximated by $(e-1/e)\varphi(k)$, furthermore the ratio $Q(k)/\varphi(k)$ goes to $e-1/e$ as $k\rightarrow \infty$. More precisely, we get the following result.
 
\begin{thm}\label{main2.5}
For every $k\geq 3$, we have 
\begin{equation}\label{f1Th2.2}
	Q(k)/\varphi(k)=e-\frac{1}{e}+ O\left( k^{-1/2}\log\log k\right). 
\end{equation}
In particular,
\[
	\lim_{k\rightarrow \infty} \frac{Q(k)}{\varphi(k)}=e-\frac{1}{e}
\]
holds.
\end{thm}

\begin{rem*}
We know that $h(A)$ is transcendental for every $A\in \mathbb{A}\cap[e^{-e},e^{1/e}]$ with $\ord(A)=\infty$ from Proposition~\ref{Prop1}. Thus we guess that $\lim_{k\rightarrow \infty} Q(k)=0$. Surprisingly, $\lim_{k\rightarrow \infty} Q(k)=\infty$ holds from Theorem~\ref{main2.5}.
\end{rem*}

\begin{thm}\label{top}
The exceptional set 
\[
\{A\in\mathbb{A}\cap [e^{-e}, e^{1/e}]\ :\ h(A) \text{ is algebraic} \}
\]
is dense in $[e^{-e},e^{1/e}]$.
\end{thm}

If we fix the order of algebraic numbers, then we can find that the exceptional set is finite by Theorem~\ref{main1}. On the other hand, we see that the union 
\[
\bigcup_{k=1}^\infty \{A\in \mathbb{A}\cap [e^{-e}, e^{1/e}]\: :\: h(A) \text{ is algebraic, and } \ord(A)=k \}
\]
is dense in  $[e^{-e}, e^{1/e}]$ from Theorem~\ref{top}. 

In the above four results, we consider only the case $x$ is positive. We can extend the iterated exponential function (\ref{IE}) to $\mathbb C$ and it is known that there exists a non-real number $x$ such that (\ref{IE}) converges. This generalization seems to give a new problem, but the following theorem holds.
\begin{thm}\label{hoyoyo} If $h(x)$ is an algebraic number, then $x$ is positive. 
\end{thm}
This theorem makes our results valid for all $\mathbb A$ with the condition that the sequence $(\ref{IE})$ converges.
We show this fact in Section \ref{section7}.

As one of generalization of these results, we also consider the case $x=\alpha^\beta$, where both $\alpha$ and $\beta$ are algebraic numbers.
If $\beta$ is not a rational number then $\alpha^\beta$ is a transcendence number by Theorem \ref{lma2}. We characterize the pair of $(\alpha,\beta)$ such that $h(\alpha^\beta)$ is an algebraic number in Section \ref{section7}.

\section{Preliminary discussion}\label{lemmas}
To prove Theorem~\ref{main1}, \ref{th:main2nd}, and  \ref{main2.5}, we show the following lemmas. 

\begin{lma}\label{lma1}
Let $x\ge 2$ be an integer, and let $a$ and $b$ be relatively prime positive integers. If $x^{a/b}$ is a positive integer, then $x^{1/b}$ is also a positive integer.
\end{lma}
\begin{proof}
Let $y=x^{a/b}$. Note that $y\in \mathbb{N}$. From the prime factorization, it follows that $x=p_1^{\alpha_1}\cdots p_n^{\alpha_n}$ and $y=p_1^{\beta_1}\cdots p_n^{\beta_n}$ for some prime numbers $p_1,\ldots, p_n$ and positive integers $\alpha_1,\ldots, \alpha_n, \beta_1,\ldots, \beta_n$. This yields that $\beta_j b= \alpha_j a$ for every $1\leq j\leq n$. From $\gcd(a,b)=1$, it is obtained that $b \mid \alpha_j$ for every $1\leq j\leq n$. Thus $x^{1/b}\in \mathbb{N}$.    
\end{proof}

\begin{lma}\label{lma1.5} 
Let $A\in \mathbb{A}\setminus \{0\}$ and $k\geq 1$. If $A^k\in\mathbb{Q}$, then we have $\ord(A)\mid k$.
\end{lma}

\begin{proof} We define $\mathbb{A}^\times$ and $\mathbb{Q}^\times$ as the multiplicative group $\mathbb{A}$ and $\mathbb{Q}$, respectively. Let $\overline{B}\in\mathbb{A}^\times/\mathbb{Q}^\times$ be the equivalent class of $B\in\mathbb{A}$. Then the cardinality of the cyclic group $\langle \overline{A}\rangle=\{\overline{}{A}^n\in\mathbb{A}^\times /\mathbb{Q}^\times\: :\: n\in\mathbb{Z} \}$ is equal to $\ord (A)$, and $(\overline{A})^k=\overline{1}$. By the theory of groups, we obtain $\ord(A)\mid k$.
\end{proof}

\begin{lma}\label{lma3}
Let $A\in \mathbb{A}\cap [e^{-e}, e^{1/e}]$. If $h(A)\in \mathbb{R}\setminus \mathbb{Q}$, then $h(A)\in \mathbb{T}$.
\end{lma}
\begin{proof}
Assume that $h(A)\in \mathbb{A}$. Then $1/h(A)\in \mathbb{A}\setminus \mathbb{Q}$ since $h(A)$ is so.  From Theorem~\ref{lma2}, $h(A)^{1/h(A)}$ is transcendental, but this is a contradiction by  (\ref{h1}).   
\end{proof}

\section{Proof of Theorem~\ref{main1}}\label{section4}

\begin{proof}[Proof of Theorem~\ref{main1}]
Fix an integer $k\geq 2$. Assume that $h(A)$ is rational. The goal of this proof is to show that $A\in \mathcal{E}(k)$ from Lemma~\ref{lma3}. It can be written as $h(A)=a/b$ for some $a,b\in \mathbb{N}$ with $\gcd(a,b)=1$. Since $\ord(A)=k$, it also can be written as $A=(x/y)^{1/k}$ for some $x,y\in \mathbb{N}$ with $\gcd(x,y)=1$.  From (\ref{h1}), 
\begin{equation}\label{C4}
	 \left(\frac{a}{b}\right)^{ \frac{b}{a} }={\left(\frac{x}{y}\right)}^{\frac{1}{k}}
\end{equation}
holds. From $\gcd(a,b)=\gcd(x,y)=1$ and (\ref{C4}), it follows that 
\begin{equation}\label{pf1.1}
a^{bk}=x^a,\quad b^{bk}=y^a.
\end{equation} 
If $x=1$, then it is easily seen that $a=1$ from (\ref{pf1.1}). Thus $y=b^{bk}$ holds. Therefore we obtain $(x/y)^{1/k}=1/b^b$. This is a contradiction to $\ord(A)=k$, which implies $x\geq 2$.
	 
Let $t=a/k$. We find that $k\mid a$ from $A^a\in \mathbb{Q}$ and Lemma~\ref{lma1.5}. Thus $t\in \mathbb{N}$ holds. We next show that $1\leq t\leq \theta\log k$. From $k\mid a$, $t$ is a positive integer, and $\gcd(t,b)=1$ holds.  From (\ref{pf1.1}), it is seen that $x^{t/b}/t=k$. From Lemma~\ref{lma1}, $x^{1/b}\in \mathbb{N}$ holds. Therefore $x$ can be written as $x=x_0^b$ for some positive integer $x_0$. We see that $x_0\neq 1$ from $x\geq 2$. Thus $x^{1/b}=x_0\geq 2$. Therefore
\[
 \frac{2^t}{t}\leq \frac{x^{t/b}}{t}=k,
\]
which implies that
\[
	 (\log2-1/e)\: t \leq \log \frac{2^t}{t} \leq \log k.
\]
Thus we have $1\leq t\leq \theta \log k$, where recall $\theta=(\log2-1/e)^{-1}$.

Let $s=b$. We find that $h(A)=a/b=kt/s$. From (\ref{h1}),  $kt/e\leq s\leq kte$ holds.

From the above discussion, it follows that 
\begin{gather*}
1\leq t\leq \theta \log k,\quad kt/e\leq s\leq tke,\quad s,t\in \mathbb{N}, \\
 x=a^{bk/a}=(kt)^{s/t},\quad y=b^{bk/a}=s^{s/t},\quad \gcd(kt,s)=1, \\
 A=\left(\frac{x}{y}\right)^{\frac{1}{k}}= \left(\frac{kt}{s}\right)^{\frac{s}{kt} }.
\end{gather*}
From Lemma~\ref{lma1} and $\gcd(kt,s)=1$, $(kt)^{1/t}$ and $s^{1/t}$ are positive integers.  Therefore $A\in \mathcal{E}(k)$.
\end{proof}

\section{Proof of Theorem~\ref{th:main2nd}}
 Let $f(y)=y^{1/y}$ on $1/e\leq y \leq e$. The function $f$ is an injection. Indeed $f'(y)=y^{1/y-2}(1-\log y) $ holds from logarithmic derivative. Therefore $f'(y)>0$ for every $y\in (1/e,e)$, which means that $f$ is an injection. Hence we immediately get the following lemma.
 
 \begin{lma}\label{lma4}
  Let $A\in \mathbb{A}\cap [e^{-e},e^{1/e}]$. If there exists $q\in \mathbb{Q}\cap [1/e,e]$ such that $A=q^{1/q}$, then $h(A)=q$. 
 \end{lma}

\begin{proof}[Proof of Theorem~\ref{main2.5}]
First we prove that the set on the left-hand side set contains the set on the right-hand.
Since $k/e\leq s \leq ke$, $(k/s)^{s/k}$ is in $[e^{-e},e^{1/e}]$. By Lemma \ref{lma4}, we have $h((k/s)^{s/k})=k/s$.
When we assume that $\mathrm{ord}(A)<k$, i.e. there is an integer $l$ such that
\begin{equation*}
1\leq l<k \quad and \quad \left(\frac{k}{s}\right)^{\frac{sl}{k}}=\frac{x}{y} \qquad (\gcd(x,y)=1),
\end{equation*}
we see that
\begin{equation*}
k^{sl}=x^k
\end{equation*}
because $\gcd(s,k)=1$ and $\gcd(x,y)=1$. Therefore $k$ divides $l$ because $\gcd(s,k)=1$ and $k$ is square free. This is a contradiction.
Hence we have $\mathrm{ord}(A)=k$.

Next we prove that the set on the left-hand side is a subset of the set on the right-hand side.
By Lemma \ref{lma3}, we can see that $h(A) \in \mathbb{Q}$.
When we put
\begin{equation*}
h(A)=\frac{a}{b} \qquad (\gcd(a,b)=1),
\end{equation*}
we can obtain
\begin{equation*}
A=\left(\frac{a}{b}\right)^{\frac{b}{a}} \qquad (a/e\leq b\leq ae).
\end{equation*}
Now we prove $a=k$.
It also can be written as $A=(x/y)^{1/k}$ with $\gcd(x,y)=1$ because $\mathrm{ord}(A)=k$.
Therefore
\begin{equation*}
\left(\frac{a}{b}\right)^{\frac{b}{a}}=\left(\frac{x}{y}\right)^{\frac{1}{k}},
\end{equation*}
and we have
\begin{equation*}
a^{kb}=x^a.
\end{equation*}
We can write $a=km \ (m\in \mathbb{N})$ and the above equation can be rewritten as
\begin{equation*}
(km)^b=x^m.
\end{equation*}
Since $k$ is square free, if a prime $p$ divides $k$ but not $m$, then $m$ divides $b$. However we put $\gcd(a,b)=1$. Hence this is a contradiction except for the case of $m=1$. If there is a prime $p$ such that 
\begin{equation*}
p \nmid k \quad \text{and} \quad p^{\alpha}\parallel m \ (\alpha \geq 1),
\end{equation*}
then $m$ divides $\alpha$ since $\gcd(b,m)=1$. Therefore $p^{\alpha}$ divides $\alpha$, but it is impossible.
Hence $k$ and $m$ are not co-prime or $m=1$.
We assume that $k$ and $m$ are not co-prime. Then there is a prime $p>2$ such that
\begin{equation*}
p^{\alpha+1} \parallel km \ (\alpha \geq 1).
\end{equation*}
Since $\gcd(b,m)=1$, $m$ divides $\alpha+1$ and therefore $p^{\alpha}$ divides $\alpha+1$. But this is impossible.
Therefore $m=1$. 
\end{proof}

For the case $k=2$, there is only one exceptional element $(2/3)^{9/2}$ in the set on the left-hand side.

\section{Proof of Theorem~\ref{main2.5}}

In order to estimate the value of $Q(k)$, we need the evaluations of arithmetic functions. Let $d(n)$ be the number of divisors of $n$, $\omega (n)$ be the number of distinct prime factors of $n$, and $\gamma$ be the Euler-Mascheroni constant.

\begin{lma} We have the following facts.\\
\cite[Theorem~2.9]{MontgomeryVaughan} For every $n\geq 3$, we obtain
\begin{equation}\label{ArithFunc3}
	\varphi (n)\geq \frac{n}{\log\log n} \left(e^{-\gamma} +O\left(\frac{1}{\log\log n}\right) \right).
\end{equation}
{\rm \cite[Theorem~2.11]{MontgomeryVaughan}} For every $n\geq 3$, 
\begin{equation} \label{ArithFunc2.1}
	\log d(n)\leq \frac{\log n}{\log\log n}\left(\log2 +O\left(\frac{1}{\log\log n}\right)\right),
\end{equation}
{\rm\cite[Theorem~3.1]{MontgomeryVaughan}} Let $P$ be a positive integer. For any $x\in\mathbb{R}$, and any $y\geq 0$,
\begin{equation}\label{ArithFunc1}
	\sum_{\substack{x<n\leq x+y \\ \gcd(n,P)=1}}1=\frac{\varphi(P)}{P}y + O\left(2^{\omega(P)}\right).
\end{equation}
\end{lma}

\begin{rem*}
From (\ref{ArithFunc2.1}), there exists $C>0$ such that for every $n\geq 3$
\begin{equation}\label{ArithFunc2}
	d(n)\leq \exp\left( \frac{C\log n}{\log\log n} \right)
\end{equation}
holds. We can take $C=1.5379$ from the result of Nicolas and Robin \cite{NicolasRobin}, but we do not use this explicit value. 

Since
$2^{\omega(P)}\leq d(P)$ holds, thus by (\ref{ArithFunc1}), we have
\begin{equation}\label{ArithFunc4}
	\sum_{\substack{x<n\leq x+y \\ \gcd(n,P)=1}}1=\frac{\varphi(P)}{P}y + O\left(d(P)\right).
\end{equation}
For every function $f(k)$ and for every non-negative function $g(k)$, we define $f(k)\ll g(k)$ as $f(k)=O(g(k))$.  
\end{rem*}

\begin{proof}[Proof of Theorem~\ref{main2.5}]
 Let $\mathcal{E}(k)$ be the set in Theorem~\ref{main1}. From Theorem~\ref{main1} and (\ref{ArithFunc4}), it follows that
\begin{align*}
	Q(k)
	&\leq \#\mathcal{E}(k) \\
	& \leq \#\{(u,t)\in \mathbb{N}^2\colon 1\leq t\leq \theta\log k,\ (kt/e)^{1/t}\leq u\leq (kte)^{1/t},\ \gcd(k,u)=1  \} \\
	&=\sum_{1\leq t\leq \theta\log k} \sum_{\substack{\quad (tk/e)^{1/t}\leq u\leq (tke)^{1/t} \\ \gcd(k,u)=1}}1 \\
	&=\sum_{1\leq t\leq \theta\log k} \left(\left((tke)^{1/t} -\left(\frac{tk}{e}\right)^{1/t}\right)\frac{\varphi(k)}{k} + O(d(k))\right)\\
	&=\left(e-\frac{1}{e}\right)\varphi(k)+ \left(\sum_{2\leq t\leq \theta\log k} \left((tke)^{1/t} -\left(\frac{tk}{e}\right)^{1/t}\right)\frac{\varphi(k)}{k}  \right)+O(d(k)\log k ).
\end{align*}
By the mean value theorem and the fact $t^{1/t}$ is bounded, the middle term is dominated by
\begin{align*}
\frac{\varphi(k)}{k} \sum_{2\leq t\leq \theta\log k} \left(ke-\frac{k}{e}\right)\frac{(k/e)^{1/t-1}}{t}\ll \varphi(k)k^{-1/2}\log\log k.
\end{align*}
By (\ref{ArithFunc3}) and (\ref{ArithFunc2}), we have
\begin{align*}
\varphi(k)k^{-1/2}\log\log k+d(k)\log k \ll \varphi(k) k^{-1/2}\log \log k.  
\end{align*}
Therefore there exists a constant $C_1>0$ such that
\[
	Q(k)/\varphi(k)-\left(e-\frac{1}{e}\right)\leq C_1  k^{-1/2}\log\log k.
\]
We next find a lower bound of $Q(k)$.  Let 
\begin{equation}\label{Ezero}
	\mathcal{E}_0(k)=\left\{ \left(\frac{k}{s}\right)^{\frac{s}{k}}\left| \begin{array}{l}
	k/e \leq s\leq ek,\ s\in\mathbb{N},\ \gcd(k,s)=1,\\
	r\mid k \text{ and } r\neq 1 \Rightarrow s^{1/r}\notin \mathbb{N} 
	\end{array}\right.\right\}
\end{equation}
for every $k\geq 3$. Then $\mathcal{E}_0(k)\subset[e^{-e},e^{1/e}]$ holds for every $k\geq 3$. Indeed, since $f(x)=x^{1/x}$ is increasing on $x\in [1/e,e]$, we have
\[
	e^{-e}\leq \left(\frac{k}{s}\right)^{\frac{s}{k}}\leq e^{1/e} 
\] 
for every $1/e\leq s/k\leq e$. Therefore $h(A)$ can be defined for every $A\in\mathcal{E}_0$. Fix $A\in \mathcal{E}_0$ and write $A=(k/s)^{s/k}$. We next show that $\ord(A)=k$. It follows that
\[
	 \left(\frac{k}{s}\right)^{\frac{s\cdot \ord(A)}{k}}=A^{\ord(A)}=\frac{x}{y}
\] 
for some relatively prime positive integers $x$ and $y$. From Lemma~\ref{lma1.5}, we obtain $\ord(A)\mid k$. Since
\[
	s^\frac{s}{k/\ord(A)}=y
\]
from $\gcd(x,y)=\gcd(k,s)=1$, it follows that $s^{\frac{1}{k/\ord(A)}} \in \mathbb{N}$ from Lemma~\ref{lma1} and $\gcd(k,s)=1$.   Therefore, the definition of $\mathcal{E}_0(k)$ leads $\ord(A)=k$. Furthermore, $h(A)$ is rational from Lemma~\ref{lma4}. Hence we get the evaluation
\[
	\#\mathcal{E}_0(k)\leq Q(k).
\]
We now find a lower bound of $\#\mathcal{E}_0(k)$. It is obtained that
\begin{align*}
\#\mathcal{E}_0(k) 
&= \sum_{\substack{k/e\leq s\leq ek \\ \gcd(k,s)=1}} 1 
- \sum_{\substack{(k/e)^{1/r}\leq s^{1/r}\leq (ek)^{1/r} \\ \gcd(k,u)=1\\s^{1/r}\in\mathbb N}} 1\\
 &\ge \sum_{\substack{k/e\leq s\leq ek \\ \gcd(k,s)=1}} 1 
- \sum_{\substack{r\mid k\\ r\neq 1}}  \sum_{\substack{(k/e)^{1/r}\leq u\leq (ek)^{1/r} \\ \gcd(k,u)=1}} 1.
\end{align*} 
From (\ref{ArithFunc4}), the first sum is equal to 
\begin{equation}\label{firstterm}
	\left(e-\frac{1}{e}\right)\varphi(k) +O\left( d(k) \right),
\end{equation}
and the second sum is equal to
\begin{align}\label{secondterm1}
\sum_{\substack{r\mid k\\ r\neq 1}} \left((ek)^{1/r}-\left(\frac{k}{e}\right)^{1/r}\right) \frac{\varphi(k)}{k}   +O(d(k)^2).
\end{align}
By the mean value theorem and the estimate (\ref{ArithFunc2}), this sum is dominated by 
\begin{align}\label{secondterm2} 
\frac{\varphi(k)}{k} \sum_{\substack{r\mid k\\ r\neq 1}}
\frac{1}{r}(k/e)^{1/r} &\ll  \frac{\varphi(k)}{k}k^{1/2}\sum_{r|k}\frac{1}{r}\leq
\frac{\varphi(k)}{k}k^{1/2} \frac{k}{\varphi(k)}=k^{1/2}.
\end{align}
Therefore, by combining (\ref{firstterm}),(\ref{secondterm1}), and (\ref{secondterm2}) , we have
\[
\# \mathcal{E}_0(k) /\varphi(k)=e-\frac{1}{e}+O(d(k)^2/\varphi(k)+k^{1/2}/\varphi(k)). 
\]
Hence, by (\ref{ArithFunc3}) and (\ref{ArithFunc2}) there exists $C_2>0$ such that
\[
-C_2k^{-1/2}\log\log k\leq Q(k)/\varphi(k) -\left(e-\frac{1}{e}\right)
\]
for every $k\geq 3$ from (\ref{ArithFunc2}). Therefore we obtain 
\[
Q(k)/\varphi(k)=e-\frac{1}{e}+O\left(k^{-1/2}\log\log k \right).
\]
Furthermore, we find that $Q(k)/\varphi(k)\rightarrow e-\frac{1}{e}$ as $k\rightarrow \infty$ from (\ref{ArithFunc3}).

%Therefore $\#\mathcal{E}_0 \geq \varphi(k)$ holds for every $k\geq k_0$. Here $\mathcal{E}_0(1)$ is a not empty set from Proposition~\ref{Prop1}. For every $k\geq 2$, there exists a prime number $p$ such that $k<p<ek$ from Bertrand's postulate. It follows that $(k/p)^{p/k}\in \mathcal{E}_0(k)$. Therefore we have  $\mathcal{E}_0(k)\neq \emptyset$ for every $1\leq k\leq k_0$. Hence we can let $C_1=\min\{\varphi(k)/\#\mathcal{E}_0(k)\: :\: 1\leq k\leq k_0  \}\cup \{1\} $, and it is obtained that
%\[
%Q(k)\geq \#\mathcal{E}_0(k)\geq C_1\varphi(k).
%\] 
\end{proof}

\begin{proof}[Proof of Theorem~\ref{top}]
Let $\mathcal{E}=\{A\in\mathbb{A}\cap [e^{-e}, e^{1/e}]\ :\ h(A) \text{ is algebraic} \}$, and let $f(x)=1/x^x$. By the definition (\ref{Ezero}), we have
\[
    \left\{ f(p/2^k)  : \ k\geq 2, \ p \text{ is odd prime}, 1/e\leq p/2^k\leq e \right\} \subseteq  \bigcup_{k=3}^\infty  \mathcal{E}_0(k) \subseteq  \mathcal {E}. 
\]
Note that the function $f(x)=1/x^x$ is a homeomorphism from $[1/e,e]$ into $[e^{-e}, e^{1/e}]$. Thus it is sufficient to show that the set
\[ 
\mathcal{F}:= \left\{ p/2^k\in \mathbb{Q}  : \ k\geq 2,\ p \text{ is odd prime} \right\}
\]
is dense in $(0,\infty)$. Here fix any real numbers $x>0$ and $\epsilon>0$. It is clear from \cite[Theorem~6.9]{MontgomeryVaughan} that if $y$ is sufficiently large real number, then there exists an odd prime number $p$ such that $p\in [y,y+y/\log y]$. Therefore if we choose a sufficiently large integer $k=k(x,\epsilon)$, then we can find an odd prime number $p$ such that
\[
(x-\epsilon)2^{k}< p< (x+\epsilon)2^{k}.    
\]
Then the following inequality holds:
\[
|x-p/2^k|<\epsilon, 
\]
which implies that $\mathcal{F}$ is dense in $(0,\infty)$. 
\end{proof}

\section{Iterated exponential on $(0,e^{-e})$}

Barrow showed that $h(x)$ does not converge on the interval $(0,e^{-e})$, but he proved that sequences of functions
\begin{equation}\label{oddexp}
x,\quad x^{x^x},\quad x^{x^{x^{x^x}}},\cdots 
\end{equation}
and
\begin{equation}\label{evenexp}
x^x,\quad x^{x^{x^x}}, \quad  x^{x^{x^{x^{x^x}}}},\cdots
\end{equation}
are convergent for every $x\in(0,e^{-e})$. We define $h_o(x)$ and $h_e(x)$ as the limits of the above sequences (\ref{oddexp}) and (\ref{evenexp}), respectively. We say that $h_o(x)$ is the odd iterated exponential function and $h_e(x)$ is the even iterated exponential function. Note that these functions can be defined on $(0,e^{-e})$. Barrow proved that
\begin{gather}\label{OddEvenEqs}
h_o(x)=x^{h_e(x)},\quad h_e(x)=x^{h_o(x)}, \quad 0<h_o(x)<\frac{1}{e}<h_e(x)<1
\end{gather}
 for every $x\in (0,e^{-e})$. We define
 \[
 R(k)=\#\{A\in \mathbb{A}\cap (0,e^{-e})\::\: \text{$h_o(A)$ and $h_e(A)$ are algebraic, and } \ord(A)=k\}.
 \]
\begin{ques}\label{finiteness} Is $R(k)$ finite? If it is true, can we find an asymptotic formula of $R(k)$?
\end{ques}

The goal of this section is to give the affirmative answer to Question~\ref{finiteness}. More precisely, we get the following results:

\begin{thm}\label{main6.1} Let $A$ be an algebraic number in the interval $(0,e^{-e})$. Then $h_o(A)$ and $h_e(A)$ are algebraic if and only if there exists a positive integer $v$ such that
\begin{equation}\label{FormulaOfA}
 A=\left(\frac{v}{v+1}\right)^{(v+1)\left(\frac{v+1}{v}\right)^{v}}.
\end{equation}
\end{thm}

From the above theorem, it follows that
\[
    R(k)=\#\left\{ v\in \mathbb{N}\ \colon\ \ord\left( \left(\frac{v}{v+1}\right)^{(v+1)\left(\frac{v+1}{v}\right)^{v}}\right)=k \right\}.
\]

\begin{thm}[the answer to Question~\ref{finiteness}]\label{main6.2}
We have
\[
R(k)=
\begin{cases}
 1 & \exists v\in \mathbb{N} \text{ {\rm s.t.} } k=v^v,\\
 0 & \text{\rm  otherwise.}
\end{cases}
\]
\end{thm}

In order to prove the results, we firstly show the following lemma:

\begin{lma}\label{OddEven}
Let $A\in \mathbb A\cap (0,e^{-e})$. If $h_o(A)$ or $h_e(A)$ is irrational, then $h_o(A)$ or $h_e(A)$ is transcendental.
\end{lma}

\begin{proof}
 If $h_o(e)$ is a transcendental number, then we immediately get this lemma. Thus we may assume that $h_o(A)$ is algebraic. It follows from (\ref{OddEvenEqs}) that $h_o(A)^{1/h_e(A)}=A$. Therefore $h_e(A)$ is transcendental from Theorem~\ref{lma2}. 
\end{proof}

 By the result of Hurwitz \cite{Hurwitz}, we obtain that
\begin{lma}\label{x^y=y^x}
The all solutions of the Diophantine equation
\begin{equation}\label{Ber1}
x^y=y^x,\quad x,y\in\mathbb{Q},\quad x>y>0
\end{equation}
are 
\begin{equation}\label{sol1}
x=(1+1/v)^{1+v},\quad y=(1+1/v)^{v} 
\end{equation}
for all $v\in \mathbb{N}$
\end{lma}

The Diophantine equation (\ref{Ber1}) appeared in a letter from Daniel Bernoulli to Christian Goldbach dated 29 June 1728. Bernoulli found that the unique solution of a pair of integers of (\ref{Ber1}), which is $(x,y)=(2,4)$. Goldbach \cite{Goldbach} responded to Bernoulli and claimed that if a pair of positive real numbers $(x,y)$ satisfies $x^y=y^x$ and $x=sy$ for some real number $s>1$, then we have  $x=s^{1/(s-1)}$ and $y=s^{s/(s-1)}$. We refer \cite{Anderson} to the reader who wants to know more details of this story.

\begin{proof}[Proof of Theorem~\ref{main6.1}]
Assume that $h_o(A)$ and $h_e(A)$ are algebraic. From Lemma~\ref{OddEven}, $h_o(A)$ and $h_e(A)$ are rational. From (\ref{OddEvenEqs}), we have 
\[
    (1/h_o(A))^{1/h_e(A)}=(1/h_e(A))^{1/h_o(A)}
\]
It follows from Lemma~\ref{x^y=y^x} that
\[
    h_o(A)=(1+1/v)^{-1-v},\quad h_e(A)=(1+1/v)^{-v}
\]
for some $v\in \mathbb{N}$. Thus the formula (\ref{FormulaOfA}) is obtained from $A=h_o(A)^{1/h_e(A)}$. 

Conversely, assume that there exists a positive integer $v$ satisfying (\ref{FormulaOfA}). Then let $x=(1+1/v)^{-1-v}$ and $y=(1+1/v)^{-v}$. We show that
\begin{gather}\label{x^(1/y)=A}
      x^{1/y}=A,\\ \label{y^(1/x)=A}
      y^{1/x}=A,\\ \label{0<x<1/e<y<1}
     0<x<\frac{1}{e}<y<1. 
\end{gather}
The formulas (\ref{x^(1/y)=A}) and (\ref{y^(1/x)=A}) are clear. It follows that $1<(1+1/v)^v<e<(1+1/v)^{1+v}$ from the Taylor expansion. Therefore we get (\ref{0<x<1/e<y<1}). From (\ref{OddEvenEqs}), we have
$h_o(A)=x=(1+1/v)^{-1-v}$ and $h_e(A)=y=(1+1/v)^{-v}$ which are algebraic.
\end{proof}

\begin{proof}[Proof of Theorem~\ref{main6.2}]
We find the solutions of the Diophantine equation:
\begin{equation}\label{Diophantine2}
\left(\frac{v}{v+1}\right)^{(v+1)\left(\frac{v+1}{v}\right)^{v}}=\left(\frac{x}{y}\right)^{1/k},\quad v,x,y\in \mathbb{N},\quad \gcd(x,y)=1.
\end{equation}
Let $A$ be the left hand side of (\ref{Diophantine2}). It follows that $A^{v^v}\in \mathbb{Q}$. From Lemma~\ref{lma1.5}, we have $k\:|\: v^v$. Let $t=v^v/k\in \mathbb{N}$. It is seen that
\[
    v^{(v+1)^{v+1}}=x^t,\quad (v+1)^{(v+1)^{v+1}}=y^t.
\]
There exists positive integers $a$ and $b$ such that 
\[
    v=a^t,\quad v+1=b^t
\]
from Lemma~\ref{lma1} and $\gcd(v,v+1)=1$. Assume that $t\geq 2$. Then it follows from $b>a$ that
\[
    1=(b-a)(b^{t-1}+b^{t-2}a+\cdots+a^{t-1})\geq t,
\]
which is a contradiction. Therefore $t=1$, which means that
\[
    k=v^v.
\]
\end{proof}

\section{Generalized case}\label{section7}
In this section, we consider $x$ is a complex number. First, we show Theorem \ref{hoyoyo}. There are many results about convergence of iterated exponential $(\ref{IE})$.
Carlsson showed that convergence of (\ref{IE}) can occur only if $x\in R=\{e^{te^{-t}} ~|~ |t|\le1\}$ in 1907 \cite{c07}. 
In 1983, Baker and Rippon showed the following theorem.
\begin{thm}[Baker and Rippon \cite{br}]
\label{brr}
Let \[\mathcal{R}=\{e^{te^{-t}} ~|~ |t| < 1\text{, or  $t$ is a root of unity}\}.\]
If $x\in \mathcal{R}$,
the sequence (\ref{IE}) converges to $e^t$. 
For almost all $t$ with $|t|=1$ the sequence (\ref{IE}) diverges.
\end{thm}
An alternative proof of Theorem \ref{brr} using Lambert's $W$ function is given by Galidakis \cite{ga05}.
%This theorem says that $x\in\mathcal R$ if and only if the sequence (\ref{IE}) for $x$ converges.
%たぶん, 微妙に必要十分条件でない気がする.(後の議論には影響しない)
In the following, we consider whether the value $h(x)$ is a transcendental number or not. The Lindemann theorem states that if $t\in\mathbb A\setminus\{0\}$ then $h(x)=e^t$ is transcendental. Therefore, $h(x)$ can be an algebraic number only if $t$ is a transcendental number. Moreover we can show a similar lemma to Lemma \ref{lma3} by the same argument of the proof of Lemma \ref{lma3}.
\begin{lma}
Let $x\in \mathbb{A}\cap \mathcal R$. If $h(x)\not\in \mathbb{Q}$, then $h(x)\in \mathbb{T}$.
\end{lma}
Since $|t|\le1$, if we assume $h(x)=e^t\in\mathbb Q$, then $t\in \mathbb R$ and $x$ is positive. Thus, there are no algebraic non-positive numbers $x$ such that $h(x)$ is an algebraic number. This shows Theorem \ref{hoyoyo}.
Also, this shows that our results can be extended to all algebraic numbers $A$ such that $h(A)$ converges.

Next we consider the case $x=\alpha^\beta$, where both of $\alpha\not=1$ and $\beta$ are algebraic numbers. Since if $\beta$ is rational then $x$ becomes algebraic, this is one of the generalizations of our results. From Theorem \ref{brr}, 
if $\beta=\frac{te^{-t}}{\log \alpha}$, where $t$ satisfies the condition of Theorem \ref{brr}, then the sequence (\ref{IE}) converges to $e^{t}$. In the following, we specify the form of $t$.
\begin{lma}
Let $\alpha\not=1,\beta$ be algebraic numbers.
For $\alpha^\beta\in \mathcal R$ it holds that $\beta h(\alpha^\beta)\not\in \mathbb{Q}$ if and only if $h(\alpha^\beta)\in \mathbb{T}$.
\end{lma}
\begin{proof}
Since we assume $\beta$ is algebraic, if $\beta h(\alpha^\beta)$ is a transcendental number then $h(\alpha^\beta)\in \mathbb T$. In the following, we assume $\beta h(\alpha^\beta)\in \mathbb A\setminus \mathbb{Q}$.
From Theorem~\ref{lma2}, $h(\alpha^\beta)=\alpha^{\beta h(\alpha^\beta)}$ is transcendental. 
This proves this lemma.
\end{proof}
This lemma shows that $h(\alpha^\beta)$ is algebraic if and only if $\beta h(\alpha^\beta)\in \mathbb Q$, that is, \[\beta h(\alpha^\beta)=\frac{t}{\log \alpha}\in\mathbb Q.\] 
Therefore, there exists an $a\in\mathbb Q$ such that $t=\log\alpha^a$. One can check easily that $\log\alpha^a$ is transcendental by the Lindemann theorem, so $\log \alpha^a$ is not a root of unity. Thus it holds that $|t|<1$, that is, \[-|\log\alpha|^{-1}< a< |\log\alpha|^{-1}.\]
We record it as a lemma.

\begin{lma}
\label{hoe}
Let $\alpha\not=1,\beta$ be algebraic numbers with $\alpha^\beta\in \mathcal R$.
Then the followings hold.
\begin{enumerate}
    \item If $h(\alpha^\beta)\in \mathbb{A}$ then $\beta=\frac{a}{\alpha^a}$, where $a\in\mathbb Q\cap(-|\log\alpha|^{-1},|\log\alpha|^{-1})$.
    \item If there exists $a\in\mathbb Q\cap(-|\log\alpha|^{-1},|\log\alpha|^{-1})$ such that $\beta=\frac{a}{\alpha^a}$, then $h(\alpha^\beta)=\alpha^a$.
\end{enumerate}
\end{lma}

Lemma \ref{hoe} implies the following theorem.
\begin{thm}
\label{abt}
Let $\alpha\not=1,\beta$ be algebraic numbers with $\alpha^\beta\in \mathcal R$.
If only one of $\ord(\alpha)$ and $\ord(\beta)$ is infinity then $h(\alpha^\beta)$ is a transcendental number. 
\end{thm}
\begin{proof}
It suffices to show that when $h(\alpha^\beta)\in \mathbb{A}$, the order of $\alpha$ is infinity if and only if the order of $\beta$ is so. First, we assume $\ord(\alpha)=k<\infty$ and $\ord(\beta)=\infty$. If $h(\alpha^\beta)\in \mathbb{A}$ then Lemma \ref{hoe} implies that there exists a rational number $a=\frac{a_1}{a_2}$ such that $\beta=\frac{a}{\alpha^a}$. Since $\ord(\alpha)=k$, $\beta^{a_2k}=\frac{a^{a_2k}}{\alpha^{a_1k}}$ is a rational number. This contradicts to $\ord(\beta)=\infty$.

Next we assume $\ord(\alpha)=\infty$ and $\ord(\beta)=k<\infty$. As well as the above, if $h(\alpha^\beta)\in \mathbb{A}$ then $\beta=\frac{a}{\alpha^a}$, that is, $\alpha=\left(\frac{a}{\beta}\right)^{\frac{1}{a}}$ for some rational number $a=\frac{a_1}{a_2}$. Then it holds that $\alpha^{a_1k}=\left(\frac{a^k}{\beta^k}\right)^{a_2}$ is rational, but this contradicts to $\ord(\alpha)=\infty$. This proves the theorem.
\end{proof}

%%%%%%%%%%%%%%%%%%%%%%%%%%%%%%%%%%%%%%%%%%%%%%%%%%%%

%\section{Future work}
%We give an open problem about iterated exponential.
%We do not give the exceptional set explicitly in this paper, the following question about this.
%\begin{ques}\label{range} Can we make the range
%\[
%	\#\mathcal{E}_0(k)\leq Q(k) \leq \#\mathcal{E}(k)
%\]
%shorter?
%\end{ques}
%%%%%%%%%%%%%%%%%%%%%%%%%%%%%%%%%%%%%%%%%%%%%%%%%%%%%%
%%%%%%%%%%%%%%%%%%%%%%%%%%%%%%%%%%%%%%%%%%%%%%%%%%%%%%

\appendix

\section{Transcendence of $h(1/\sqrt[n]{n})$}
We do not mention an example of $A\in\mathbb{A}$ such that $h(A)$ is transcendental. Thus this appendix gives such an example. 

\begin{prop} \label{main2}
For every $n\ge 2$, $h(1/\sqrt[n]{n})$ is transcendental.
\end{prop}

\begin{rem*}\label{remark2}
Let $f(x)=x^x$ on $(0,1)$.  From the logarithmic derivative, $f'(x)=x^x(\log x+1)$ holds. Therefore $f(e^{-1})=e^{-1/e}$ is the minimum value of $f$ on $(0,1/2)$. It follows that $e^{-1/e}\leq f(x)\leq 1$, which implies that $h(f(x))$ is convergent for every $x\in (0,1)$. Hence $h((1/n)^{1/n})$ can be defined for all $n\geq 2$. 
\end{rem*}

\begin{proof}[Proof of Proposition~\ref{main2}]
Fix $n\geq 2$. From Lemma~\ref{lma3}, it is sufficient to show that  $h(1/\sqrt[n]{n})$ is not rational. Thus we assume that  $h(1/\sqrt[n]{n})$ is rational. It can be written as $h(1/\sqrt[n]{n})=a/b$ for some relatively prime positive integers $a,b$. From (\ref{h1}), it follows that
\[
	\left(\frac{a}{b}\right)^{\frac{b}{a}}=\left(\frac{1}{n}\right)^{\frac{1}{n}},
\]
which implies that $a^{bn}=1$ and $b^{bn}=n^a$. Thus $a=1$ holds. Since $n\neq1$, we have $b\neq 1$. Therefore it is obtained that
\[
n< 2^n< 2^{2n}\leq b^{bn}=n.
\]
This is a contradiction.
\end{proof}

It is well known that $h(\sqrt{2})=2$. Indeed we see that $\sqrt{2}\in [e^{-e},e^{1/e}]$ from the calculation, and $h(\sqrt{2})^{1/h(\sqrt{2})}=\sqrt{2}$. Here $2^{1/2}=\sqrt{2}$ also holds. Therefore we have $h(\sqrt{2})=2$ from Lemma~\ref{lma4}. On the other hand, $h(1/\sqrt{2})$ is transcendental from Proposition~\ref{main2} with $n=2$.

\section*{Acknowledgements}
The second author would like to thank Koshi Furuzono and Naru Yamaguchi for discussing and advising.
This work was supported by JSPS KAKENHI Grant Numbers 19J20878 and 19J10705.

\end{document}